\documentclass{amsart}
\usepackage{amssymb,amsmath,amsthm,amscd}
\usepackage{mathtools}
\usepackage{mathrsfs}
\usepackage{mathabx}
\usepackage[T1]{fontenc}
\usepackage{hyperref}
\hypersetup{colorlinks,linkcolor={red},citecolor={blue}} 
\usepackage{enumerate}
\usepackage{comment}
\usepackage[final,notref,notcite]{showkeys}
\usepackage{tikz}
\usetikzlibrary{decorations.pathreplacing}
\usetikzlibrary{calc}
\usepackage{esint}

\makeatletter
\@namedef{subjclassname@2020}{%
	\textup{2020} Mathematics Subject Classification}
\makeatother

\newtheorem{theorem}{Theorem}
\newtheorem{lemma}[theorem]{Lemma}
\newtheorem{corollary}[theorem]{Corollary}

\newtheorem{definition}[theorem]{Definition}

\newtheorem{remark}[theorem]{Remark}
\newtheorem{proposition}[theorem]{Proposition}

\numberwithin{theorem}{section}
\numberwithin{equation}{section}

\def\N {{\mathbb N}}
\def\Z {{\mathbb Z}}

\def\R {{\mathbb R}}

\DeclareMathOperator{\diff}{d\!}
\DeclareMathOperator{\supp}{supp}

\DeclareMathOperator{\lac}{lac}

\DeclarePairedDelimiter\abs{\lvert}{\rvert}
\DeclarePairedDelimiter\norm{\lVert}{\rVert}
\def\<{\left\langle}
\def\>{\right\rangle}

\title{Duality for outer $L^{p,\infty}$ spaces}

\author{Marco Fraccaroli}
\email{mfraccar@math.uni-bonn.de}


\begin{document}
	
\date{\today}
	
\begin{abstract}	
	The $L^{p,\infty}$ quasi-norm of functions on a measure space can be characterized in terms of their pairing with normalized characteristic functions. We generalize this result to the case of the outer $L^{p,\infty}$ quasi-norms for appropriate sizes. This characterization provides an explicit form for certain implicitly defined quasi-norms appearing in an article of Di Plinio and Fragkos.  
\end{abstract}

\maketitle

\section{Introduction}

The characterization of the classical $L^{p,\infty}$ quasi-norm (or weak $L^p$ quasi-norm) of functions on a measure space $(X,\omega)$ in terms of their pairing with appropriately normalized characteristic functions is a standard result in analysis. On one hand, for every $p \in (1,\infty]$ it exhibits a K\"{o}the duality between the $L^{p,\infty}$ spaces and the collection of characteristic functions endowed with the $L^{p',1}$ quasi-norm analogous to the K\"{o}the duality between the $L^{p}$ and $L^{p'}$ spaces. On the other hand, for every $p \in (0,1]$ the pairing is more subtle. 

To state both of the characterizations, for every measure space $(X,\omega)$ we introduce the following auxiliary notations. First, we define $\Sigma$ to be the collection of $\omega$-measurable subsets of $X$ and $\mathcal{M}$ to be the collection of $\Sigma$-measurable functions on $X$. Next, we define the collection $\Sigma_\omega \subseteq \Sigma$ by
\begin{equation*}
\Sigma_\omega \coloneq \Big\{ A \in \Sigma \colon \omega(A) \notin \{ 0, \infty \}   \Big\},
\end{equation*} 
and for every measurable subset $A \in \Sigma_\omega$ we define the collection $\Sigma''_\omega(A) \subseteq \Sigma$ by
\begin{equation*}
\Sigma''_\omega(A) \coloneq \Big\{ B \in \Sigma \colon B \subseteq A, \omega(B) \geq \frac{\omega(A)}{2} \Big\}.
\end{equation*}
We have the following characterization of the $L^{p,\infty}$ quasi-norms.\footnote{Throughout the article the notation $\alpha \sim_C \beta $ with $\alpha, \beta, C > 0$ means that $\alpha \leq C \beta$ and $\beta \leq C \alpha$.}

\begin{theorem} [e.g. Lemma~2.6 in \cite{MR3052499}] \label{thm:base}
For every $p \in (0, \infty]$ there exists a constant $C = C(p)$ such that for every measure space $(X,\omega)$ the following properties hold true.

\begin{enumerate} [(i)]
	\item For every $p > 1$ and every function $f \in L^{p,\infty}(X,\omega)$ we have
\begin{equation*}
\norm{f}_{L^{p,\infty}(X,\omega)} \sim_{C} \sup \Big\{ \omega(A)^{\frac{1}{p} - 1} \norm{f 1_A}_{L^1(X,\omega)} \colon A \in \Sigma_\omega \Big\}. 
\end{equation*}

\item For every $p \leq 1$ and every function $f \in L^{p,\infty}(X,\omega)$ we have
\begin{equation*}
\norm{f}_{L^{p,\infty}(X,\omega)} \sim_{C} \sup \Big\{ \inf \Big\{ \omega(A)^{\frac{1}{p} - 1} \norm{f 1_B}_{L^1(X,\omega)} \colon B \in \Sigma''_\omega(A) \Big\} \colon A \in \Sigma_\omega \Big\}.
\end{equation*}
\end{enumerate}
\end{theorem}

In this article, we are interested in proving an analogous result in the case of the outer $L^{p,\infty}$ quasi-norm of functions on an outer measure space introduced by Do and Thiele in \cite{MR3312633}. We refer to the article of Do and Thiele \cite{MR3312633}, the articles and the Ph.D. thesis of the author \cite{MR4292789,2023Fra,Fraccaroli} for a more detailed introduction to the theory of outer $L^p$ spaces. However, to make this article self-contained, we briefly recall some minimal definitions. 

In particular, we are interested in the case of the outer $L^{p,\infty}$ quasi-norms on a $\sigma$-finite setting $(X,\mu,\omega)$ defined as follows. First, we recall that an \emph{outer measure} $\mu$ on a set $X$ is a function from $\mathcal{P}(X)$, the collection of subsets of $X$, to $[0,\infty]$ satisfying the following properties:
\begin{enumerate} [(i)]
	\item $\mu(\varnothing) = 0$.
	\item For all subsets $A, B \subseteq X$ such that $A \subseteq B$ we have $\mu(A) \leq \mu(B)$.
	\item For every countable collection $\{ A_n \colon n \in \N \} \subseteq \mathcal{P}(X)$ we have
	\begin{equation*}
	\mu \Big( \bigcup_{n \in \N} A_n \Big) \leq \sum_{n \in \N} \mu(A_n).
	\end{equation*}
\end{enumerate}

\begin{definition}[$\sigma$-finite setting]
	Let $X$ be a set, let $\omega$ be a $\sigma$-finite measure on $X$, and let $\mu$ be a $\sigma$-finite outer measure on $X$. Namely, we assume there exist two countable collections $\{ A_n \colon n \in \N \}, \{ B_n \colon n \in \N \} \subseteq \mathcal{P}(X)$ such that
	\begin{gather*}
	X = \bigcup_{n \in \N} A_n = \bigcup_{n \in \N} B_n, \\
	\forall n \in \N, \qquad \qquad \omega(A_n) < \infty, \qquad \mu(B_n) < \infty.
	\end{gather*}
	We assume that for every subset $A \subseteq X$ we have
	\begin{equation*}
	\mu(A) = 0 \Rightarrow \omega(A) = 0.
	\end{equation*} 
	We define $(X,\mu,\omega)$ a \emph{$\sigma$-finite setting}.
\end{definition}

Next, let $\mathcal{A} \subseteq \{ A \in \Sigma \colon \mu(A) \notin \{0, \infty\} \}$ be a collection of measurable subsets such that
\begin{equation*}
X = \bigcup_{A \in \mathcal{A}} A.
\end{equation*}
We recall that a \emph{size} $S = (S,\mathcal{A})$ on the collection $\mathcal{M}$ of $\Sigma$-measurable functions on $X$ is a function from $\mathcal{M} \times \mathcal{A}$ to $[0,\infty]$ satisfying the following properties:
\begin{enumerate} [(i)]
	\item For every $\lambda \in \R$, every function $f \in \mathcal{M}$, and every subset $A \in \mathcal{A}$ we have
	\begin{equation*}
	S( \lambda f)(A) = \abs{\lambda} S(f)(A).
	\end{equation*}
	\item For all functions $f,g \in \mathcal{M}$ such that $\abs{f} \leq \abs{g}$ $\omega$-almost everywhere and every subset $A \in \mathcal{A}$ we have
	\begin{equation*}
	S(f)(A) \leq S(g)(A).
	\end{equation*}
	\item There exists a constant $C \in [1,\infty)$ such that for all functions $f,g \in \mathcal{M}$ and every subset $A \in \mathcal{A}$ we have
	\begin{equation*}
	S(f+g)(A) \leq C [ S(f)(A) + S(g)(A) ].
	\end{equation*} 
\end{enumerate}
In particular, every size $(S,\mathcal{A})$ satisfies the following weaker property.
\begin{enumerate} [(i)]
	\item[(iv)] There exists a constant $K \in [1,\infty)$ such that for every measurable function $f$ on $X$, every subset $A \in \mathcal{A}$, and every measurable subset $B \in \Sigma$ we have
	\begin{equation} \label{eq:K_suport}
	S(f)(A) \leq K [ S(f 1_B)(A) + S(f 1_{B^c})(A) ].
	\end{equation}
	In particular, the constant $K$ is smaller than or equal to the constant $C$ appearing in property~(iii). 
\end{enumerate}
For example, we have the sizes $\ell^\infty_\omega$ and $\ell^r_\omega$ with $r \in (0,\infty)$ on the collection $\mathcal{A} = \{ A \in \Sigma \colon \mu(A) \notin \{0,\infty\} \}$ defined by
\begin{equation} \label{eq:sizes}
\begin{aligned}
\ell^\infty_\omega(f)(A) & \coloneq \norm{f 1_A}_{L^{\infty}(X,\omega)}, \\
\ell^r_\omega(f)(A) & \coloneq \mu(A)^{-\frac{1}{r}} \norm{f 1_A}_{L^r(X,\omega)}.
\end{aligned}
\end{equation}
It is easy to observe that the size $(\ell^1_\omega, \mathcal{A})$ satisfies the inequality in \eqref{eq:K_suport} with constant $K = 1$. In fact, it satisfies the equality.

For every size $S = (S,\mathcal{A})$ we define the outer $L^\infty_\mu(S)$ spaces by the following quasi-norms on functions $f \in \mathcal{M}$
\begin{equation*}
\norm{f}_{L^\infty_\mu(S)} = \norm{f}_{L^{\infty,\infty}_\mu(S)} \coloneq \sup \Big\{ S(f)(A) \colon A \in \mathcal{A} \Big\}.
\end{equation*}
Next, for every $\lambda \in (0,\infty)$ and every function $f \in \mathcal{M}$ we define
\begin{equation*}
\mu(S (f) > \lambda) \coloneq \inf \Big\{ \mu(A) \colon A \in \Sigma, \norm{f 1_{A^c}}_{L^\infty_\mu(S)} \leq \lambda \Big\}.
\end{equation*}
Finally, for all $p \in (0,\infty)$, $q \in (0,\infty]$ we define the outer Lorentz $L^{p,q}_\mu(S)$ spaces by the following quasi-norms on functions $f \in \mathcal{M}$
\begin{equation*}
\norm{f}_{L^{p,q}_\mu(S)} \coloneq p^{\frac{1}{q}} \norm[\Big]{ \lambda \mu(S (f) > \lambda)^{\frac{1}{p}} }_{L^q((0,\infty), \diff \lambda / \lambda)},
\end{equation*}
where the exponent $q^{-1}$ for $q=\infty$ is understood to be $0$. In fact, we can define the same quasi-norms via the decreasing rearrangement function $f^{\ast}$ on $[0,\infty)$ associated with $f$ defined by
\begin{equation*}
f^{\ast} (t) \coloneq \inf \Big\{ \alpha \in [0,\infty) \colon \mu(S(f) > \alpha) \leq t \Big\}, 
\end{equation*}
where the infimum over an empty collection is understood to be $\infty$. In particular, we have
\begin{equation*}
\norm{f}_{L^{p,q}_\mu(S)} = \norm[\Big]{ t^{\frac{1}{p}} f^{\ast}(t) }_{L^q((0,\infty), \diff t / t)}.
\end{equation*}

To state the main result of this article, for every $\sigma$-finite setting $(X,\mu,\omega)$ we introduce the following auxiliary notations. We define the collection $\Sigma_\mu \subseteq \Sigma$ by
\begin{equation*}
\Sigma_\mu \coloneq \Big\{ A \in \Sigma \colon \mu(A) \notin \{0 , \infty \} \Big\},
\end{equation*} 
and for every measurable subset $A \in \Sigma_\mu$ we define the collection $\Sigma'_\mu(A) \subseteq \Sigma$ by
\begin{equation} \label{eq:correct}
\Sigma'_\mu(A) \coloneq \Big\{ B \in \Sigma \colon B \subseteq A, \mu(A \setminus B) \leq \frac{\mu(A)}{2} \Big\}.
\end{equation}
We are ready to state the main theorem of this article. To improve its readability, we omit explicit expressions for the constants. However, in the proofs we keep track of their dependence on the parameters.

\begin{theorem} \label{thm:intro_first}
	For all $a, p, q, r \in (0, \infty]$ there exists a constant $C = C(a,p,q,r)$ such that for every $\sigma$-finite setting $(X,\mu,\omega)$ the following properties hold true.
	
	\begin{enumerate} [(i)]
		\item For every $p > a$ and every function $f \in L^{p,\infty}_\mu(\ell^r_\omega)$ we have
		\begin{equation*}
		\norm{f}_{L^{p,\infty}_\mu(\ell^r_\omega)} \sim_C \sup \Big\{ \mu(A)^{\frac{1}{p} - \frac{1}{a}} \norm{f 1_A}_{L^{a,q}_\mu(\ell^r_\omega)} \colon A \in \Sigma_\mu \Big\}.
		\end{equation*}
		
		\item For every $p \leq a$ and every function $f \in L^{p,\infty}_\mu(\ell^r_\omega)$ we have 
		\begin{align*}
		\norm{f}_{L^{p,\infty}_\mu(\ell^r_\omega)} \sim_C \sup \Big\{ \inf \Big\{ \mu(A)^{\frac{1}{p} -\frac{1}{a}} \norm{f 1_B}_{L^{a,q}_\mu(\ell^r_\omega)} \colon B \in \Sigma'_\mu(A) \Big\} \colon A \in \Sigma_\mu \Big\}.
		\end{align*}
	\end{enumerate}
\end{theorem}

We point out some comments about the statement of Theorem~\ref{thm:intro_first}. First, for $p \neq \infty$ we can generalize property~$(i)$ to the case of arbitrary sizes, see Theorem~\ref{thm:intro_second}. Moreover, for $p = \infty$ we have a slight generalization of property~$(i)$ too, see Remark~\ref{rmk:suff}. In particular, the arguments used to prove Theorem~\ref{thm:intro_second} are different from those appearing in a previous article of the author \cite{MR4292789}, where we characterized the outer $L^p_\mu(\ell^r_\omega)$ quasi-norm in terms of the outer $L^1_\mu(\ell^1_\omega)$ pairing with arbitrary functions with normalized outer $L^{p'}_\mu(\ell^{r'}_\omega)$ quasi-norm. The difference is due to the fact that in Theorem~\ref{thm:intro_first} and Theorem~\ref{thm:intro_second} we have the same size $S$ both in the outer $L^{p,\infty}(S)$ quasi-norm to be characterized and the outer $L^{a,q}(S)$ quasi-norm used to measure the pairing. Therefore, we rely only on the quasi-subadditivity of sizes, namely property~(iv) of sizes, to compare the super level measures
\begin{equation*}
\mu(S(f) > \rho), \qquad \qquad \mu(S(f 1_A) > \lambda),
\end{equation*}
where $A \in \Sigma$ is an appropriate measurable subset and $\rho, \lambda \in (0,\infty)$.

Next, we compare the collections $\Sigma''_\omega(A)$ and $\Sigma'_\mu(A)$ appearing in the statement of property~$(ii)$ in Theorem~\ref{thm:base} and Theorem~\ref{thm:intro_first} respectively. By the additivity on disjoint sets of the measure $\omega$, in the case of the collection $\Sigma''_\omega(A)$ we have the equality
\begin{equation*}
\Sigma''_\omega(A) = \Sigma'_\omega(A) \coloneq \Big\{ B \in \Sigma \colon B \subseteq A, \omega(A \setminus B) \leq \frac{\omega(A)}{2} \Big\} .
\end{equation*}
However, in the case of the collection $\Sigma'_\mu(A)$ in general we only have the inclusion
\begin{equation*}
\Sigma'_\mu(A) \subseteq \Sigma''_\mu(A) \coloneq \Big\{ B \in \Sigma \colon B \subseteq A, \mu(B) \geq \frac{\mu(A)}{2} \Big\}.
\end{equation*} 
In fact, in the characterization of the outer $L^{p,\infty}$ quasi-norm stated in property~$(ii)$ in Theorem~\ref{thm:intro_first} we cannot replace $\Sigma'_\mu(A)$ with $\Sigma''_\mu(A)$. We clarify this claim in Remark~\ref{rmk:rmk_2}.

Finally, the interest in the characterization of outer $L^{p,\infty}$ quasi-norms by pairing with normalized characteristic functions was sparked by the function spaces appearing in Subsection~3.8 in the article of Di Plinio and Fragkos \cite{2022arXiv220408051D}. The latter ones, denoted by $X^{p,q}_a(J,\kappa,\mathsf{size}_{2,\star})$, are quasi-normed spaces of functions on the collection $\mathcal{H}$ of dyadic Heisenberg tiles in $\R^2$. In particular, the set $\mathcal{H}$ is the collection of dyadic rectangles of area $1$ in $\R^2$. The spaces introduced by Di Plinio and Fragkos are associated with new implicitly defined quasi-norms. More specifically, the new quasi-norm of a function $f$ is defined in terms of the pairing of $f$ with appropriately normalized characteristic functions. In particular, both the pairing and the renormalization of the characteristic functions are measured in terms of outer Lorentz $L^{p,q}$ spaces with respect to an appropriate size. As a corollary of Theorem~\ref{thm:intro_second}, we show that the $X^{p,q}_a(J,\kappa,\mathsf{size}_{2,\star})$ quasi-norms are equivalent to explicit outer $L^{p,\infty}$ ones with respect to an appropriate size, see Corollary~\ref{thm:corollary}. We provide all the definitions to state this equivalence in Section~\ref{sec:corollary}. Here we briefly point out that the sizes $(S^r,\mathcal{A})$ with $r \in [1,\infty)$ appearing in the article of Di Plinio and Fragkos differ from those defined in \eqref{eq:sizes} in the following way, namely for every subset $A \in \mathcal{A}$ we have 
\begin{equation*}
S^r(f)(A) \coloneq \sup \{ \ell^r_\omega(f)(\widetilde{A}) \colon \widetilde{A} \in \widetilde{\mathcal{A}} \},
\end{equation*}
where $\widetilde{\mathcal{A}}$ is an appropriate collection of strict subsets of $A$. The size $(S^1,\mathcal{A})$ satisfies the inequality in \eqref{eq:K_suport} with constant $K = 1$, but it does not satisfy the equality. This difference makes it more difficult to use decomposition arguments akin to those for the size $\ell^1_\omega$ appearing in the previous article of the author \cite{MR4292789}.

We briefly comment about the use of the spaces $X^{p,q}_a(J,\kappa,\mathsf{size}_{2,\star})$ in the article of Di Plinio and Fragkos \cite{2022arXiv220408051D}. In \cite{2022arXiv220408051D} the main focus of the authors is about improving the $L^p$ estimates for the Carleson maximal operator $\mathcal{C}$ defined for every Schwartz function $f \in \mathcal{S}(\R)$ and every $x \in \R$ by
\begin{equation*}
\mathcal{C} f (x) \coloneq \sup \Big\{ \abs[\Big]{\int_{\xi \leq N} \widehat{f}(\xi) e^{i x \xi} \diff \xi }  \colon N \in \R \Big\}.
\end{equation*}
In particular, Di Plinio and Fragkos are interested in the case when the exponent is close to the endpoint case $p=1$. The previously known best estimate was a restricted weak-type $L^p$ one of the following form. For every $p \in (1,\infty)$, every subset $\abs{F} \subseteq \R$, and every measurable function $f$ such that $\abs{f} \leq 1_F$ we have
\begin{equation*}
\norm{ \mathcal{C} f }_{L^{p,\infty} (\R)} \leq \frac{C p^2}{p-1} \abs{F}^{\frac{1}{p}},
\end{equation*}
where the constant $C$ is independent of $p$. This bound was due to the work of Hunt \cite{MR0238019} building on the seminal article of Carleson \cite{MR199631}. An equivalent estimate was later proved by Lacey and Thiele in \cite{MR1783613} in the language of generalized restricted weak-type $L^p$ bounds.

The improvement obtained by Di Plinio and Fragkos is a proof of a weak-type $L^p$ estimate of the following form. For every $p \in (1,2]$ and every Schwartz function $f \in \mathcal{S}(\R)$ we have
\begin{equation*}
\norm{ \mathcal{C} f }_{L^{p,\infty} (\R)} \leq \frac{C}{p-1} \norm{f}_{L^{p} (\R)},
\end{equation*}
where the constant $C$ is independent of $p$. This result (Theorem~A in \cite{2022arXiv220408051D}) follows from a sparse norm bound for the Carleson maximal operator (Theorem~B in \cite{2022arXiv220408051D}). In fact, Di Plinio and Fragkos prove a sparse norm bound for the Carleson maximal operator associated with H\"{o}rmander-Mihlin multipliers. We refer to the article of Di Plinio and Fragkos for both the definition of a sparse norm bound and the description of the family of weighted estimates for the operator implied by the sparse norm bound.

The proof of the sparse norm bound follows from a standard stopping forms argument and two inequalities involving the spaces $X^{p,q}_a(J,\kappa,\mathsf{size}_{2,\star})$. The role played by the two inequalities is analogous to the two-step programme outlined in the article of Do and Thiele \cite{MR3312633}. First, a version of H\"{o}lder's inequality for outer $L^p$ spaces. Next, the boundedness of the Carleson embedding map from classical to outer $L^{p}$ spaces.  

In fact, the first inequality is a generalized outer H\"{o}lder's inequality for outer Lorentz $L^{p,q}$ spaces on the $\sigma$-finite setting $(E(J),\mu_\kappa,\omega_\kappa)$ (Proposition~3.9 in \cite{2022arXiv220408051D}). In particular, for all $ a, p_1 \in (1,\infty)$, $ a \leq p_1$ and $p_2, \dots, p_m \in [1,\infty)$ such that
\begin{equation*}
\varepsilon \coloneq \Big( \sum_{j=1}^m \frac{1}{p_j} \Big) - 1 > 0,
\end{equation*} 
we have
\begin{equation} \label{eq:Thm3.9}
\begin{aligned}
\norm[\Big]{\prod_{j=1}^m F_j}_{L^1(E(J),\omega_\kappa)} \leq \frac{C a}{\varepsilon (a-1)} & \norm{F_1}_{X^{p_1,\infty}_a(J,\kappa,\mathsf{size}_{2,\star})} \times  \\
& \times \prod_{j=2}^m \max \{ \norm{F}_{L^{p_j,\infty}_{\mu_\kappa}(s_j)}, \norm{F}_{L^{\infty}_{\mu_\kappa}(s_j)} \},
\end{aligned}
\end{equation}
where the sizes $\mathsf{size}_{2,\star}, s_2, \dots, s_m$ satisfy an appropriate condition. We provide all the necessary definitions in Section~\ref{sec:corollary}. In view of the equivalence stated in Corollary~\ref{thm:corollary} we obtain Proposition~3.9 in \cite{2022arXiv220408051D} as corollary of standard interpolation results in the theory of outer $L^p$ spaces, see Appendix~\ref{sec:appendix}. 

The second inequality is the boundedness of the Carleson embedding map $W$ measured in the $X^{tp',\infty}_2(J,\kappa,\mathsf{size}_{2,\star})$ spaces with a constant independent of $p$ (Theorem~D in \cite{2022arXiv220408051D}). For every Schwartz function $f \in \mathcal{S}(\R)$ the embedded function $W[f] \colon \mathcal{H} \to [0,\infty)$ is defined on every $H = I_H \times J_H  \in \mathcal{H}$ by
\begin{equation*}
W[f](H) \coloneq \sup \Big\{ \abs{ \langle f, \phi \rangle } \colon \phi \in \Phi(H) \Big\},
\end{equation*}
where $\Phi(H)$ is an appropriate collection of $L^1$-normalized smooth enough functions adapted to $I_H$ and such that $\supp \widehat{\phi} \subseteq J_H$. Then, for all $t \in (1,\infty)$, $p \in (1,2]$, every dyadic interval $J$, every collection $\mathcal{P} \subseteq \mathcal{H}$ of dyadic Heisenberg tiles, and every Schwartz function $f \in \mathcal{S}(\R)$ we have
\begin{equation} \label{eq:ThmD}
\norm{W[f] 1_{\mathcal{P}}}_{X^{tp',\infty}_2(J,\kappa,\mathsf{size}_{2,\star})} \leq C(t,\kappa) [f]_{p,\mathcal{P}}.
\end{equation}
In the previous display, the quasi-norm $[f]_{p,\mathcal{P}}$ measures the minimal value on $I_H$ for every $H \in \mathcal{P}$ of the maximal $L^p$ average of $f$. In view of the equivalence stated in Corollary~\ref{thm:corollary} the estimate in \eqref{eq:ThmD} is a weak version of the boundedness of the Carleson embedding map $W$, namely
\begin{equation*}
\norm{W[f] 1_{\mathcal{P}}}_{L^{tp',\infty}_{\mu_\kappa}(s_\kappa)} \leq C(t,\kappa) [f]_{p,\mathcal{P}},
\end{equation*}
for an appropriate size $s_\kappa$. It should be compared with two previously known results. First, the weak-type estimate for the Carleson embedding map $W$ proved by Do and Thiele in \cite{MR3312633}, namely
\begin{equation*}
\norm{W[f] 1_{\mathcal{P}}}_{L^{2,\infty}_{\mu_\kappa}(s_\kappa)} \leq C(\kappa) [f]_{2,\mathcal{P}}.
\end{equation*}
Next, the strong-type estimate for the Carleson embedding map $W$ applied to the localization of the function $f$ proved by Di Plinio and Ou in \cite{MR3829613}, namely
\begin{equation*}
\norm{W[f 1_{3J}] 1_{\mathcal{P}}}_{L^{tp'}_{\mu_\kappa}(s_\kappa)} \leq C(t,\kappa,p) [f]_{p,\mathcal{P}}.
\end{equation*}

\subsection*{Guide to the article}
In Section~\ref{sec:preliminaries} we introduce additional definitions, we make some auxiliary observations, and in Theorem~\ref{thm:intro_second} we generalize property~$(i)$ in Theorem~\ref{thm:intro_first} to the case of arbitrary sizes. In Section~\ref{sec:first_proof} we prove Theorem~\ref{thm:intro_second}. After that, in Section~\ref{sec:other_proof} we prove the remaining properties stated in Theorem~\ref{thm:intro_first}. Finally, in Section~\ref{sec:corollary} we introduce the setting on the collection $\mathcal{H}$ of dyadic Heisenberg tiles in $\R^2$ and all the relevant notions. Moreover, we state and prove the result about the spaces introduced by Di Plinio and Fragkos in Corollary~\ref{thm:corollary}. In view of this result, we detail the alternative proof of Proposition~3.9 in \cite{2022arXiv220408051D} in Appendix~\ref{sec:appendix}. In Appendix~\ref{sec:appendix_2} we provide the proofs of the auxiliary results stated in Section~\ref{sec:other_proof}. 

\section*{Acknowledgements}
The author gratefully acknowledges financial support by the CRC~1060 \emph{The Mathematics of Emergent Effects} at the University of Bonn, funded through the Deutsche Forschungsgemeinschaft. 

The work was initiated during the HIM Trimester Program \emph{Interactions between Geometric measure theory, Singular integrals, and PDE} held at the Hausdorff Research Institute for Mathematics in Bonn during the first quarter of 2022. The author gratefully acknowledges the organizers of the program, as well as the academic and supporting staff of the Institute for their hospitality.

The author is thankful to Francesco Di Plinio, Anastasios Fragkos, and Christoph Thiele for helpful discussions, comments and suggestions that greatly improved the exposition of the material.

\section{Preliminaries} \label{sec:preliminaries}

\subsection*{Auxiliary definitions and generalization of Theorem~\ref{thm:intro_first}}
Let $(X,\mu,\omega)$ be a $\sigma$-finite setting. First, for every measurable subset $A \in \Sigma$ we define the collection $\Sigma_\mu(A) \subseteq \Sigma$ by
\begin{equation*}
\Sigma_\mu(A) \coloneq \Big\{ B \in \Sigma \colon B \subseteq A, \mu(B) \notin \{0 , \infty \} \Big\},
\end{equation*}
the collection $\Sigma^{\circ}_{\omega}(A) \subseteq \Sigma$ by
\begin{equation*}
\Sigma^{\circ}_{\omega}(A) \coloneq \Big\{ B \in \Sigma \colon B \subseteq A, \omega(A \setminus B) = 0 \Big\},
\end{equation*}
and the value $\mu^{\circ}(A) \in [0,\infty]$ by
\begin{equation*}
\mu^{\circ}(A) \coloneq \inf \Big\{ \mu(B) \colon B \in \Sigma^{\circ}_\omega(A) \Big\}.
\end{equation*}

Next, for every $r \in (0,\infty)$ and every size $(S,\mathcal{A})$ we define the sizes $(S_r,\mathcal{A})$, $(S_\infty,\mathcal{A})$ as follows. For every $A \in \mathcal{A}$ and every measurable function $f$ on $X$ we define
\begin{align} \label{eq:size_r}
S_r(f)(A) & \coloneq (S(f^r)(A))^{\frac{1}{r}}, \\  \label{eq:size_infty}
S_\infty(f)(A) & \coloneq \lim_{r \to \infty} (S(f^r)(A))^{\frac{1}{r}}.
\end{align}
It is easy to observe that
\begin{equation*}
S_\infty(f)(A) = \norm{f 1_A}_{L^\infty(X,\omega)},
\end{equation*}
hence for all $p,q \in (0,\infty]$ and every measurable function $f$ on $X$ we have
\begin{equation*}
\norm{f}_{L^{p,q}_\mu(S_\infty)} = \norm{f}_{L^{p,q}_\mu(\ell^\infty_\omega)}.
\end{equation*}
Moreover, it is easy to observe that for every $r \in (0,\infty]$ we have
\begin{equation*}
(\ell^r_\omega, \Sigma_\mu) = ((\ell^1_\omega)_r, \Sigma_\mu).
\end{equation*}

Finally, we state a generalization of property~$(i)$ in Theorem~\ref{thm:intro_first} to the case of arbitrary sizes.
\begin{theorem} \label{thm:intro_second}
	For all $a, q, r \in (0, \infty]$, every $p \in (0,\infty)$, and every $K \in [1,\infty)$ there exists a constant $C = C(a,p,q,r,K)$ such that for every $\sigma$-finite setting $(X,\mu,\omega)$  and every size $(S,\mathcal{A})$ on $X$ satisfying the condition in \eqref{eq:K_suport} with constant $K$ the following property holds true.	
	\begin{enumerate} [(i)]
		\item For every $p > a$, $p \neq \infty$ and every function $f \in L^{p,\infty}_\mu(S_r)$ we have
		\begin{equation*}
		\norm{f}_{L^{p,\infty}_\mu(S_r)} \sim_C \sup \Big\{ \mu(A)^{\frac{1}{p} - \frac{1}{a}} \norm{f 1_A}_{L^{a,q}_\mu(S_r)} \colon A \in \Sigma_\mu \Big\}.
		\end{equation*}
	\end{enumerate}
\end{theorem}

\subsection*{Auxiliary observations}
Let $(X,\mu,\omega)$ be a $\sigma$-finite setting. First, we observe that for all $p \in (0,\infty)$, $q \in (0,\infty]$ and every measurable subset $A \in \Sigma$ we have 
\begin{equation} \label{eq:first_equation}
\norm{1_A}_{L^{p,q}_\mu(\ell^\infty_\omega)} = p^{\frac{1}{q}} \norm[\Big]{ \lambda \mu ( \ell^\infty_\omega ( 1_A ) > \lambda )^{\frac{1}{p}}  }_{L^q ( (0, \infty), \diff \lambda / \lambda )} = \Big( \frac{p}{q} \Big)^{\frac{1}{q}} \mu^{\circ}(A)^{\frac{1}{p}}, 
\end{equation}
where $(p/q)^{\frac{1}{q}}$ for $q = \infty$ is understood to be $1$.

Next, we state and prove the following auxiliary result.

\begin{lemma} \label{thm:lemma_aux}
	Let $(X,\mu,\omega)$ be a $\sigma$-finite setting and let $(S,\mathcal{A})$ be a size on $X$ satisfying the condition in \eqref{eq:K_suport} with constant $K \in [1,\infty)$.
	
	For every $r \in (0,\infty)$, every $\rho \in (0,\infty)$, every $\delta \in [0,1]$, every measurable function $f$ on $X$, and every measurable subset $\Omega \in \Sigma$ such that
	\begin{equation} \label{eq:size_condition}
	\norm{f 1_{\Omega^c}}_{L^\infty_\mu(S_r)} \leq K^{-\frac{1}{r}} (1 - \delta^r)^{\frac{1}{r}} \rho,
	\end{equation}
	we have
	\begin{equation*} 
	\mu( S_r(f) > \rho ) \leq \mu \Big( S_r(f 1_{\Omega}) > K^{-\frac{1}{r}} \delta \rho \Big).
	\end{equation*}
\end{lemma}

\begin{proof}
	We argue by contradiction and we assume that
	\begin{equation*}
	\mu( S_r(f) > \rho ) > \mu \Big( S_r(f 1_{\Omega}) > K^{-\frac{1}{r}} \delta \rho \Big).
	\end{equation*}
	In particular, there exists a measurable subset $\Theta \in \Sigma$, $\Theta \subseteq \Omega$ such that
	\begin{equation*}
	\norm{f 1_{\Omega} 1_{\Theta^c}}_{L^\infty_\mu(S_r)} \leq K^{- \frac{1}{r}} \delta \rho, \qquad \qquad \mu(\Theta) < \mu(S_r(f) > \rho). 
	\end{equation*}
	Therefore, there exists a subset $A \in \mathcal{A}$ such that
	\begin{align*}
	S (f^r 1_{\Omega} 1_{\Theta^c} ) (A) & = S_r (f 1_{\Omega} 1_{\Theta^c} ) (A)^r \leq K^{-1} \delta^r \rho^r , \\
	S (f^r 1_{\Theta^c}) (A) & = S_r (f 1_{\Theta^c}) (A)^r > \rho^r.
	\end{align*}
	By the inequality in \eqref{eq:K_suport}, we have
	\begin{equation*}
	S_r( f 1_{\Omega^c} )(A)^r = S(f^r 1_{\Omega^c} )(A) > K^{-1} (1 - \delta^r) \rho^r .
	\end{equation*}
	Together with the condition in \eqref{eq:size_condition}, the inequality in the previous display yields a contradiction.
\end{proof}

\section{Proof of Theorem~\ref{thm:intro_second}} \label{sec:first_proof}

We prove Theorem~\ref{thm:intro_second} under the assumption $K = 1$. For arbitrary $K \in [1,\infty)$ we apply the same arguments only obtaining a different constant $C$.

Moreover, we make some additional assumptions. At the end of the proof we comment on the modifications needed to drop them.

\begin{itemize}
	\item For every arbitrary function $f \in L^{p,\infty}_\mu(S_r)$ and every $p \in (0,\infty)$ we assume that there exists $\rho \in (0,\infty)$ such that
	\begin{equation} \label{eq:add_1}
	\norm{f}_{L^{p,\infty}_\mu(S_r)} = \rho \mu( S_r (f) > \rho )^{\frac{1}{p}} \in (0,\infty).
	\end{equation}

	\item For every arbitrary $\lambda \in (0,\infty)$ we assume that there exists a measurable subset $\Omega_\lambda \in \Sigma$ such that
	\begin{equation} \label{eq:add_2}
	\norm{f 1_{\Omega^c_\lambda}}_{L^{\infty}_\mu(S_r)} \leq \lambda, \qquad \qquad \mu(\Omega_\lambda) = \mu( S_r (f) > \lambda ) \in [0,\infty).
	\end{equation}
	
	\item For every arbitrary measurable subset $A \in \Sigma$ we assume that there exists a measurable subset $B \in \Sigma^{\circ}_\omega(A)$ such that
	\begin{equation} \label{eq:add_3}
	\mu(B) = \mu^{\circ}(A) = \inf \Big\{ \mu(D) \colon D \in \Sigma^{\circ}_\omega(A) \Big\}.
	\end{equation}

\item We make the additional assumption that $a = 1$ and we define $p' \in (-\infty,0) \cup [1, \infty] $ by 
\begin{equation*}
1 = \frac{1}{p} + \frac{1}{p'}.
\end{equation*}
Therefore, for $p \in [1,\infty]$ we have $p' \in [1,\infty]$ and for $p \in (0,1)$ we have $p' \in (-\infty,0)$.
\end{itemize}

\begin{proof} [Proof of Theorem~\ref{thm:intro_second}]
	Without loss of generality we assume $f \nequiv 0$, otherwise the property is trivially satisfied.
	
	For every $p \in (1,\infty]$ we have $p' \neq \infty$. Therefore, for every $q \in (0,\infty]$ the inequality
	\begin{equation*}
	\sup \Big\{ \mu(A)^{- \frac{1}{p'}} \norm{f 1_A}_{L^{1,q}_\mu(S_r)} \colon A \in \Sigma_\mu \Big\} \leq 2 \Big( \frac{p'}{q} \Big)^{\frac{1}{q}} \norm{f}_{L^{p,\infty}_\mu(S_r)} ,
	\end{equation*}
	where $(p'/q)^{\frac{1}{q}}$ for $q = \infty$ is understood to be $1$, follows from outer H\"{o}lder's inequality for outer Lorentz $L^{p,q}$ spaces (Theorem~\ref{thm:Holder_Lorentz} in the Appendix) and the equality in \eqref{eq:first_equation}. 
	
	We turn to the proof of the inequality
	\begin{equation*}
	\norm{f}_{L^{p,\infty}_\mu(S_r)} \leq C \sup \Big\{ \mu(A)^{- \frac{1}{p'}} \norm{f 1_A}_{L^{1,q}_\mu(S_r)} \colon A \in \Sigma_\mu \Big\} .
	\end{equation*} 
	We distinguish two cases.
	
	{\textbf{Case I: $r= \infty$.}} Without loss of generality, we assume
	\begin{equation*}
	\Omega_\rho \subseteq \Big\{ x \in X \colon \abs{f(x)} > \rho \Big\}.
	\end{equation*}
	By the inclusion in the previous display and the equality in \eqref{eq:first_equation}, for every $q \in (0,\infty]$ we have
	\begin{equation*}
	\norm{f 1_{\Omega_{\rho} }}_{L^{1,q}_\mu(\ell^\infty_\omega)}  \geq \norm{\rho 1_{\Omega_{\rho}}}_{L^{1,q}_\mu(\ell^\infty_\omega)} = q^{-\frac{1}{q}} \rho \mu( \Omega_{\rho} )  ,
	\end{equation*}
	where $q^{\frac{1}{q}}$ for $q = \infty$ is understood to be $1$. Together with the assumptions in \eqref{eq:add_1}--\eqref{eq:add_3}, the previous inequality yields the desired inequality.

	{\bf{Case II: $r \neq \infty$.}} By Lemma~\ref{thm:lemma_aux}, for every $\delta \in [0,1]$ we have
	\begin{equation*} 
	\mu( S_r(f) > \rho ) \leq \mu \Big( S_r(f 1_{\Omega_{(1-\delta^r)^{1/r} \rho}}) > \delta \rho \Big).
	\end{equation*}
	By the assumption on $\rho$ in \eqref{eq:add_1}, we have
	\begin{equation*}
	 \mu(\Omega_{(1-\delta^r)^{1/r} \rho}) = \mu\Big( S_r(f) > (1-\delta^r)^{\frac{1}{r}} \rho \Big) \leq (1-\delta^r)^{-\frac{p}{r}} \mu( S_r(f) > \rho).
	\end{equation*}
	By the inequalities in the previous two displays, for every $q \in (0,\infty]$ we have
	\begin{equation*}
	\norm{f 1_{\Omega_{(1-\delta^r)^{1/r} \rho}} }_{L^{1,q}_{\mu}(S_r)} \geq \delta (1- \delta^r)^{\frac{p}{r p'}} \mu(\Omega_{(1-\delta^r)^{1/r} \rho})^{\frac{1}{p'}} \rho \mu(S_r(f) > \rho)^{\frac{1}{p}}  ,
	\end{equation*}
	Together with the assumptions in \eqref{eq:add_1}--\eqref{eq:add_3}, the previous inequality yields the desired inequality for every $q \in (0,\infty]$
	\begin{equation*}
	\norm{f}_{L^{p,\infty}_\mu(S_r)} \leq C \sup \Big\{ \mu(A)^{- \frac{1}{p'}} \norm{f 1_A}_{L^{1,q}_\mu(S_r)} \colon A \in \Sigma_\mu \Big\},
	\end{equation*}
	where the minimal value of the constant $C$ for $q \in (0,\infty]$ is defined by
	\begin{equation*}
	\inf \Big\{ \delta^{-1} (1-\delta^r)^{-\frac{p}{r p'}} \colon \delta \in [0,1] \Big\} = p^{\frac{p}{r}} (p-1)^{\frac{1-p}{r}} .
	\end{equation*}
\end{proof}

We comment on the modifications needed to drop the additional assumptions.
\begin{itemize}
	\item For every function $f \in L^{p,\infty}_\mu(S_r)$ and every $\delta > 0$ there exists $\rho_\delta \in (0,\infty)$ such that 
	\begin{equation*}
	\rho_\delta \mu( S_r (f) > \rho_\delta ) \leq \norm{f}_{L^{p,\infty}_\mu(S_r)} \leq (1+\delta) \rho_\delta \mu( S_r (f) > \rho_\delta ).
	\end{equation*}
	
	\item For every function $f \in L^{p,\infty}_\mu(S_r)$, every $\lambda \in (0,\infty)$, and every $\delta > 0$ there exists a measurable subset $\Omega_{\lambda,\delta} \in \Sigma$ such that 
	\begin{equation*}
	\norm{f 1_{\Omega^c_{\lambda,\delta}}}_{L^{\infty}_\mu(S_r)} \leq \lambda, \qquad \qquad \mu(\Omega_{\lambda,\delta}) \leq (1+\delta) \mu( S_r (f) > \lambda ).
	\end{equation*}
	
	\item For every measurable subset $A \in \Sigma$ and every $\delta > 0$ there exists a measurable subset $B_{\delta} \in \Sigma^{\circ}_\omega(A)$ such that
	\begin{equation*}
	\mu(B) \leq (1+\delta) \mu^{\circ}(A).
	\end{equation*}
	In particular, for every measurable subset $A \in \Sigma$ we have
	\begin{equation*}
	\mu(A) \geq \mu^{\circ}(A) = \norm{1_{A}}_{L^{1,q}_\mu(\ell^\infty_\omega)}.
	\end{equation*}
	
	\item For all $a,p,q,r \in (0,\infty]$, every setting $(X,\mu,\omega)$, and every function $f \in L^{p,q}_\mu(S_r)$ we have
	\begin{equation*}
	\norm{f}_{L^{p,q}_\mu(S_r)} = \norm{f^{a}}^{\frac{1}{a}}_{L^{p/a,q/a}_\mu(S_{r/a})}.
	\end{equation*}
	In particular, for every constant $C = C\Big( 1, \frac{p}{a}, \frac{q}{a}, \frac{r}{a} \Big)$ such that
	\begin{equation*}
	\norm{f}_{L^{p/a,\infty}_\mu(S_{r/a})} \leq C \sup \Big\{ \mu(A)^{\frac{a}{p} - 1} \norm{f 1_A}_{L^{1,q/a}_\mu(S_{r/a})} \colon A \in \Sigma_\mu \Big\},
	\end{equation*}
	we have
	\begin{equation*}
	\norm{f}_{L^{p,\infty}_\mu(S_r)} \leq C^{\frac{1}{a}} \sup \Big\{ \mu(A)^{\frac{1}{p} - \frac{1}{a}} \norm{f 1_A}_{L^{a,q}_\mu(S_r)}  \colon A \in \Sigma_\mu \Big\},
	\end{equation*}
	and analogously for an inequality in the opposite direction.
\end{itemize}
For every fixed $\delta > 0$ we can apply the arguments we described in the previous proofs. Taking $\delta > 0$ arbitrarily small we obtain a proof of Theorem~\ref{thm:intro_second} with the appropriate constants.

\section{Proof of Theorem~\ref{thm:intro_first}} \label{sec:other_proof}

We prove Theorem~\ref{thm:intro_second} under the additional assumptions listed at the beginning of the previous section. Moreover, we make some other additional assumptions. At the end of the proof we comment on the modifications needed to drop them.

\begin{itemize}
	\item For every arbitrary function $f \in L^{\infty}_\mu(\ell^r_\omega)$ we assume that there exists $A_\rho \in \mathcal{A} \subseteq \Sigma_\mu$ such that
\begin{equation} \label{eq:add_1_bis}
\norm{f}_{L^{\infty}_\mu(\ell^r_\omega)} = \ell^r_\omega(f)(A_\rho) = \rho \in (0,\infty).
\end{equation}

\item For every arbitrary function $f \in L^{p,\infty}_\mu(\ell^r_\omega)$, every arbitrary measurable subset $B \in \Sigma$, and every arbitrary $\lambda \in (0,\infty)$ we assume that there exists a measurable subset $\Omega_\lambda(B) \in \Sigma$, $\Omega_\lambda (B) \subseteq B$ such that
\begin{equation} \label{eq:add_5}
\norm{f 1_B 1_{\Omega_\lambda(B)^c}}_{L^{\infty}_\mu(\ell^r_\omega)} \leq \lambda, \qquad \qquad \mu(\Omega_\lambda(B)) = \mu( \ell^r_\omega (f 1_B) > \lambda ) \in [0,\infty).
\end{equation}
\end{itemize}

Next, we state and prove the following auxiliary results. We postpone their proofs to Appendix~\ref{sec:appendix_2}.

\begin{lemma} \label{thm:lemma_2}
	Let $(X,\mu,\omega)$ be a $\sigma$-finite setting. Let $r \in (0,\infty)$. Let $f \in L^\infty_\mu(\ell^r_\omega)$. Let $A \in \Sigma_\mu$ be a measurable subset such that 
	\begin{equation} \label{eq:condition_3}
	\ell^r_\omega(f)(A) \geq \rho \in (0,\infty).
	\end{equation}
	Then for every $\lambda \in (0,\rho)$ we have
	\begin{equation*}
	\mu(A) \leq \frac{\norm{f}^r_{L^\infty_\mu(\ell^r_\omega)}}{\rho^r - \lambda^r} \mu(\ell^r_\omega(f) > \lambda). 
	\end{equation*}
\end{lemma}

\begin{lemma} \label{thm:lemma_1}
	Let $(X,\mu,\omega)$ be a $\sigma$-finite setting. Let $p,r \in (0,\infty]$. Let $f \in L^{p,\infty}_\mu(\ell^r_\omega)$. For every $\lambda \in (0,\infty)$ such that
	\begin{equation} \label{eq:condition}
	\norm{f}_{L^\infty_\mu(\ell^r_\omega)} > \lambda,
	\end{equation}
	there exists a measurable subset $B \in \Sigma$ such that
	\begin{equation} \label{eq:condition_2}
	\ell^r_\omega(f)(B) > \lambda, \qquad \qquad \norm{f 1_{B^c}}_{L^\infty_\mu(\ell^r_\omega)} \leq \lambda.
	\end{equation}
\end{lemma}

\begin{proof} [Proof of property~$(i)$ in Theorem~\ref{thm:intro_first}]
	Without loss of generality we assume $f \nequiv 0$, otherwise the property is trivially satisfied.
	
	We turn to the proof of the inequality
	\begin{equation*}
	\norm{f}_{L^{\infty}_\mu(\ell^r_\omega)} \leq C \sup \Big\{ \mu(A)^{- 1} \norm{f 1_A}_{L^{1,q}_\mu(\ell^r_\omega)} \colon A \in \Sigma_\mu \Big\} .
	\end{equation*} 
	We distinguish two cases.
	
	{\textbf{Case I: $r = \infty$.}} See {\textbf{Case I}} in the proof of Theorem~\ref{thm:intro_second}.
	
	{\textbf{Case II: $r \neq \infty$.}} By Lemma~\ref{thm:lemma_2}, for every $k \in (0, 1)$ we have
	\begin{equation} \label{eq:second_equation_2}
	\mu(A_{\rho}) \leq (1-k^r)^{-1} \mu( \Omega_{k \rho}(A_{\rho}) ).
	\end{equation}
	By the previous inequality, for every $q \in (0,\infty]$ and every $k \in (0,1)$ we have 
	\begin{equation*}
	\norm{f 1_{A_{\rho}}}_{L^{1,q}_\mu(\ell^r_\omega)} \geq (k-k^{r+1}) \rho \mu(A_{\rho}). 
	\end{equation*}	
	Together with the assumptions in \eqref{eq:add_1}--\eqref{eq:add_3} and in \eqref{eq:add_1_bis}--\eqref{eq:add_5}, the previous inequality yields the desired inequality for every $q \in (0,\infty]$
	\begin{equation*}
	\norm{f}_{L^{\infty}_\mu(\ell^r_\omega)} \leq C \sup \Big\{ \mu(A)^{- 1} \norm{f 1_A}_{L^{1,q}_\mu(\ell^r_\omega)} \colon A \subseteq X, A \neq \varnothing \Big\},
	\end{equation*}
	where the minimal value of the constant $C$ is defined by
	\begin{equation*}
	\inf \Big\{ (k-k^{r+1})^{-1} \colon  k \in (0,1) \Big\} = (r+1)^{\frac{r+1}{r}} r^{-1} .
	\end{equation*}	
\end{proof}

\begin{remark} \label{rmk:suff}
	A sufficient condition on the size $(S,\mathcal{A})$ to generalize both Lemma~\ref{thm:lemma_2} and property~$(i)$ in Theorem~\ref{thm:intro_first} to the case of the sizes $(S_r,\mathcal{A})$  is the following.
	
	For every measurable function $f$ on $X$, every subset $A \in \mathcal{A}$, and every measurable subset $B \in \Sigma$ we have
	\begin{equation} \label{eq:new}
	S(f)(A) \mu(A) \leq \norm{f}_{L^\infty_\mu(S)} \mu( B) + \norm{f 1_{A \setminus B}}_{L^\infty_\mu(S)} \mu(A ) .
	\end{equation}
	
	It is easy to observe that the size $(\ell^1_\omega, \Sigma_\mu)$ satisfies the inequality in \eqref{eq:new}.
\end{remark}

\begin{proof} [Proof of property~$(ii)$ in Theorem~\ref{thm:intro_first}]
	Without loss of generality we assume $f \nequiv 0$, otherwise the property is trivially satisfied.
	
	Let $K'$ be defined by
	\begin{equation*}
	K' \coloneq \sup \Big\{ \inf \Big\{ \mu(A)^{-\frac{1}{p'}} \norm{f 1_B}_{L^{1,q}_\mu(\ell^r_\omega)} \colon B \in \Sigma'_\mu(A) \Big\} \colon A \in \Sigma_\mu \Big\}.
	\end{equation*}	
	
	{\textbf{Case I: $r= \infty$.}} 
	We start with the proof of the inequality
	\begin{equation*}
	K' \leq C \norm{f}_{L^{p,\infty}_{\mu}(\ell^\infty_\omega)}.
	\end{equation*}
	For every measurable subset $A \in \Sigma_\mu$ we define $\lambda_0 \in (0,\infty)$ by
	\begin{equation*}
	\lambda_0 \coloneq 2^{\frac{1}{p}} \mu(A)^{-\frac{1}{p}} \norm{f}_{L^{p,\infty}_{\mu}(\ell^\infty_\omega)}.
	\end{equation*}
	By the definition of the outer $L^{p,\infty}_\mu(\ell^\infty_\omega)$ quasi-norm we have
	\begin{equation*}
	\mu(\Omega_{\lambda_0}) \leq \lambda_0^{-p} \norm{f}_{L^{p,\infty}_{\mu}(\ell^\infty_\omega)}^p \leq \frac{\mu(A)}{2}, 
	\end{equation*}
	hence we define the measurable subset $B \in \Sigma'_\mu(A)$ by
	\begin{equation*}
	B \coloneq A \setminus \Omega_{\lambda_0} \subseteq \Big\{ x \in A \colon \abs{f(x)} < \lambda_0 \Big\}.
	\end{equation*}
	By the inclusion in the previous display and the equality in \eqref{eq:first_equation}, for every $q \in (0,\infty]$ we have
	\begin{align*}
	\mu(A)^{- \frac{1}{p'}} \norm{f 1_B}_{L^{1,q}_\mu(\ell^\infty_\omega)} & \leq 2^{\frac{1}{p}} \mu(A)^{- 1} \norm{f}_{L^{p,\infty}_{\mu}(\ell^\infty_\omega)} \norm{1_B}_{L^{1,q}_\mu(\ell^\infty_\omega)} \\
	& \leq 2^{\frac{1}{p}} q^{-\frac{1}{q}} \norm{f}_{L^{p,\infty}_\mu(\ell^\infty_\omega)}  .
	\end{align*}
	where $q^{\frac{1}{q}}$ for $q=\infty$ is understood to be $1$. Together with the assumptions in \eqref{eq:add_1}--\eqref{eq:add_3}, the previous inequality yields the desired inequality.
	
	We turn to the proof of the inequality
	\begin{equation*}
	\norm{f}_{L^{p,\infty}_{\mu}(\ell^\infty_\omega)} \leq C K'.
	\end{equation*}
	Without loss of generality, we assume
	\begin{equation} \label{eq:1}
	\Omega_\rho \subseteq \Big\{ x \in X \colon \abs{f(x)} > \rho \Big\}.
	\end{equation}
	For every measurable subset $B_{\rho} \in \Sigma'_\mu( \Omega_{\rho} )$ we have
	\begin{equation*}
	\mu(B_\rho) \geq \frac{\mu(\Omega_\rho)}{2}.
	\end{equation*}
	By the inclusion in \eqref{eq:1}, the equality in \eqref{eq:first_equation}, and the previous inequality, for every $q \in (0,\infty]$ we have
	\begin{equation*}
	\norm{f 1_{B_\rho}}_{L^{1,q}(\ell^\infty_\omega)} \geq \norm{\rho 1_{B_\rho}}_{L^{1,q}(\ell^\infty_\omega)} = q^{-\frac{1}{q}} \rho \mu(B_\rho),
	\end{equation*}
	where $q^{\frac{1}{q}}$ for $q=\infty$ is understood to be $1$. Together with the assumptions in \eqref{eq:add_1}--\eqref{eq:add_3}, the previous two inequalities yield the desired inequality.

	{\textbf{Case II: $r \neq \infty$.}} We start with the proof of the inequality
	\begin{equation*}
	K' \leq C \norm{f}_{L^{p,\infty}_{\mu}(\ell^r_\omega)}.
	\end{equation*}
	For every measurable subset $A \in \Sigma_\mu$ we define $\lambda_0 \in (0,\infty)$ by
	\begin{equation*}
	\lambda_0 \coloneq 2^{\frac{1}{p}} \mu(A)^{-\frac{1}{p}} \norm{f}_{L^{p,\infty}_{\mu}(\ell^r_\omega)}.
	\end{equation*}
	By the definition of the outer $L^{p,\infty}_\mu(\ell^r_\omega)$ quasi-norm we have
	\begin{equation*}
	\mu(\Omega_{\lambda_0}) \leq \lambda_0^{-p} \norm{f}_{L^{p,\infty}_{\mu}(\ell^r_\omega)}^p \leq \frac{\mu(A)}{2}. 
	\end{equation*}
	and we define the measurable subset $B \in \Sigma'_\mu(A)$ by
	\begin{equation*}
	B \coloneq A \setminus \Omega_{\lambda_0}.
	\end{equation*}
	By outer H\"{o}lder's inequality for outer Lorentz $L^{p,q}$ spaces (Theorem~\ref{thm:Holder_Lorentz} in Appendix~\ref{sec:appendix}) and the equality in \eqref{eq:first_equation}, for every $q \in (0,\infty]$ we have
	\begin{align*}
	\mu(A)^{- \frac{1}{p'}} \norm{f 1_B}_{L^{1,q}_\mu(\ell^r_\omega)} & \leq 2 \mu(A)^{- \frac{1}{p'}} \norm{f 1_{\Omega_{\lambda_0}^c}}_{L^{\infty}_{\mu}(\ell^r_\omega)} \norm{1_B}_{L^{1,q}_\mu(\ell^\infty_\omega)} \\
	& \leq 2^{1 + \frac{1}{p}} q^{-\frac{1}{q}} \norm{f}_{L^{p,\infty}_{\mu}(\ell^r_\omega)}  ,
	\end{align*}
	where $q^{\frac{1}{q}}$ for $q=\infty$ is understood to be $1$. Together with the assumptions in \eqref{eq:add_1}--\eqref{eq:add_3}, the previous inequality yields the desired inequality.

	We turn to the proof of the inequality
	\begin{equation*}
	\norm{f}_{L^{p,\infty}_{\mu}(\ell^r_\omega)} \leq C K'.
	\end{equation*}
	By the definition of the outer $L^{p,\infty}_\mu(\ell^r_\omega)$ quasi-norm, the choice of $\rho$, and the subadditivity of $\mu$, for every $M \in (1,\infty)$ we have
	\begin{align*}
	M \rho \mu( \Omega_{M \rho} )^{\frac{1}{p}} & \leq \rho \mu( \Omega_{\rho} )^{\frac{1}{p}} , \\
	\mu( \Omega_{\rho} ) & \leq \mu( \Omega_{M \rho} ) + \mu( \Omega_\rho ( \Omega_{M \rho}^c ) ) .
	\end{align*}
	Hence, we have
	\begin{equation} \label{eq:aux_3}
	0 < M^{p} \mu( \Omega_{\rho} ) - \mu( \Omega_{\rho} ) \leq M^{p} [ \mu( \Omega_{\rho} ) - \mu( \Omega_{M \rho} ) ] \leq M^{p} \mu( \Omega_\rho ( \Omega_{M \rho}^c ) ) .
	\end{equation}
	The previous chain of inequalities yields
	\begin{equation} \label{eq:aux_1}
	\norm{f}_{L^{p,\infty}_\mu(\ell^r_\omega)} \leq M (M^{p} - 1)^{-\frac{1}{p}} \rho \mu( \Omega_\rho ( \Omega_{M \rho}^c ) )^{\frac{1}{p}}.
	\end{equation}
	Moreover, by the chain of inequalities in \eqref{eq:aux_3} we have $\Omega_\rho(\Omega^c_{M \rho}) \neq \varnothing$. In particular, for every $\varepsilon > 0$ we have
	\begin{equation*}
	\norm{f 1_{\Omega^c_{M \rho}}}_{L^\infty_\mu(\ell^r_\omega)} > \rho_\varepsilon,
	\end{equation*}
	where $\rho_\varepsilon \in (0,\rho)$ is defined by
	\begin{equation*}
	\rho_\varepsilon \coloneq (1 - \varepsilon) \rho.
	\end{equation*}
	Therefore, by Lemma~\ref{thm:lemma_1} there exists a measurable subset $B_{\rho_\varepsilon} \in \Sigma_\mu( \Omega_{M \rho}^c)$ such that
	\begin{equation} \label{eq:sel_1}
	\ell^r_\omega(f)(B_{\rho_\varepsilon}) > \rho_\varepsilon , \qquad \qquad \norm{f 1_{\Omega_{M \rho}^c} 1_{B_{\rho_\varepsilon}^c} }_{L^\infty_\mu(\ell^r_\omega)} \leq \rho_\varepsilon.
	\end{equation}
	By the inequality in \eqref{eq:aux_1} and the second inequality in \eqref{eq:sel_1}, we have
	\begin{equation} \label{eq:second_equation_bis}
	\norm{f}_{L^{p,\infty}_\mu(\ell^r_\omega)} \leq M (M^{p} - 1)^{-\frac{1}{p}} \rho \mu( B_{\rho_\varepsilon} )^{\frac{1}{p}}  .
	\end{equation}
	
	By the definition of $K'$, there exists a subset $D_{\rho_\varepsilon} \in \Sigma'_\mu(B_{\rho_\varepsilon})$ such that
	\begin{equation} \label{eq:auxiliary}
	\norm{f 1_{D_{\rho_\varepsilon}}}_{L^{1,q}(\ell^r_\omega)} \leq K' \mu(B_{\rho_\varepsilon})^{ \frac{1}{p'}}.
	\end{equation}
	Furthermore, by the first inequality in \eqref{eq:sel_1} and the inclusion $B_{\rho_\varepsilon} \subseteq \Omega^c_{M \rho}$, for every subset $D \in \Sigma'_\mu(B_{\rho_\varepsilon})$ we have
	\begin{align*}
	\ell^r_\omega(f)(D) & \geq \mu(D)^{-\frac{1}{r}} [ \rho_\varepsilon^r \mu(B_{\rho_\varepsilon}) - M^r \rho^r \mu(B_{\rho_\varepsilon} \setminus D) ]^{\frac{1}{r}} \\
	& \geq \mu(B_{\rho_\varepsilon})^{-\frac{1}{r}} \Big[ \rho_\varepsilon^r \mu(B_{\rho_\varepsilon}) - M^r \rho^r \frac{\mu(B_{\rho_\varepsilon})}{2} \Big]^{\frac{1}{r}} \\
	& \geq \Big( 1 - \frac{M_\varepsilon^r }{2  } \Big)^{\frac{1}{r}} \rho_\varepsilon ,
	\end{align*}
	where $M_\varepsilon \in (M,\infty)$ is defined by
	\begin{equation*}
	M_\varepsilon \coloneq \frac{M}{1-\varepsilon}.
	\end{equation*}
	In particular, for every $\lambda = k \Big( 1 - \frac{M_\varepsilon^r }{2} \Big)^{1/r} \rho_\varepsilon$ with $k \in (0,1)$ we have
	\begin{equation} \label{eq:second_equation_ter}
	\mu(B_{\rho_\varepsilon}) \leq 2 \mu(D) \leq 2 M_\varepsilon^r  \Big( 1 - \frac{M_\varepsilon^r }{2} \Big)^{-1} (1-k^r)^{-1} \mu( \Omega_\lambda ( D ) ) ,
	\end{equation}
	where the second inequality follows by Lemma~\ref{thm:lemma_2}. By the second inequality in \eqref{eq:second_equation_ter}, for every $q \in (0,\infty)$ and every $k \in (0,1)$ we have
	\begin{equation*}
	\norm{f 1_D}_{L^{1,\infty}_\mu(\ell^r_\omega)} \geq M_\varepsilon^{-r} \Big( 1 - \frac{M_\varepsilon^r}{2 } \Big)^{ \frac{r+ 1}{r}} (k - k^{r+1}) \rho_\varepsilon \mu(D) .
	\end{equation*}
	Taking $\varepsilon > 0$ arbitrarily small, together with the assumptions in \eqref{eq:add_1}--\eqref{eq:add_3} and in \eqref{eq:add_5}, the inequality in \eqref{eq:second_equation_bis}, the first inequality in \eqref{eq:second_equation_ter}, and the inequality in \eqref{eq:auxiliary}, the previous inequality yields the desired inequality for every $q \in (0,\infty]$
	\begin{equation*}
	\norm{f}_{L^{p,\infty}_\mu(\ell^r_\omega)} \leq 2 C K',
	\end{equation*}
	where the minimal value of the constant $C$ for every $q \in (0,\infty]$ is defined by
	\begin{align*}
	\inf & \Big\{ ( 1 - M^{-p} )^{-\frac{1}{p}} \Big( 1 - \frac{M^r}{2} \Big)^{-  \frac{r+ 1}{r}} M^r ( k - k^{r+1} )^{-1} \colon M \in (1,2^{\frac{1}{r}}), k \in (0,1) \Big\} = \\
	& = \frac{(r+1)^{\frac{r+1}{r}}}{r} \inf \Big\{ M^{r} (1 - M^{-p})^{-\frac{1}{p}} \Big( 1 - \frac{M^r}{2} \Big)^{-  \frac{r+ 1}{r}} \colon M \in (1,2^{\frac{1}{r}}) \Big\} .
	\end{align*}
\end{proof}

We comment on the modifications needed to drop the additional assumption.
\begin{itemize}
	\item For every function $f \in L^{\infty}_\mu(\ell^r_\omega)$ and every $\delta > 0$ there exists $A_{\rho,\delta} \in \mathcal{A}$ such that 
	\begin{equation*}
	(1- \delta) \rho \leq \ell^r_\omega(f)(A_{\rho,\delta}) \leq \rho = \norm{f}_{L^{\infty}_\mu(\ell^r_\omega)}.
	\end{equation*}
\end{itemize}
For every fixed $\delta > 0$ we can apply the arguments we described in the previous proofs. Taking $\delta > 0$ arbitrarily small we obtain a proof of Theorem~\ref{thm:intro_first} with the appropriate constants.

\begin{remark} \label{rmk:rmk_2}
	For every measurable subset $A \in \Sigma$, we recall the definition of the collection $\Sigma''_\mu(A) \subseteq \Sigma$ by
	\begin{equation*}
	\Sigma''_\mu(A) \coloneq \Big\{ B \in \Sigma \colon B \subseteq A, \mu(B) \geq \frac{\mu(A)}{2} \Big\}.
	\end{equation*} 
	Let $K''$ be defined by
	\begin{equation*}
	K'' \coloneq \sup \Big\{ \inf \Big\{ \mu(A)^{\frac{1}{p}-\frac{1}{a}} \norm{f 1_B}_{L^{a,q}_\mu(\ell^r_\omega)} \colon B \in \Sigma''_\mu(A) \Big\} \colon A \in \Sigma_\mu \Big\}.
	\end{equation*}	
	For every measurable subset $B \in \Sigma'_\mu (A)$ we have $\mu(A) \leq 2 \mu(B)$, hence
	\begin{equation*}
	K'' \leq K'.
	\end{equation*}
	In fact, for every $p \leq a$ and every function $f \in L^{p,\infty}_\mu(\ell^\infty_\omega)$ we have
	\begin{equation*}
	\norm{f}_{L^{p,\infty}_{\mu}(\ell^\infty_\omega)} \sim_{C} K''.
	\end{equation*}
	Moreover, for all $p \leq a$, $r \neq \infty$ and every function $f \in L^{p,\infty}_\mu(\ell^r_\omega)$ we have
	\begin{equation*}
	K'' \leq C \norm{f}_{L^{p,\infty}_{\mu}(\ell^r_\omega)} .
	\end{equation*}
	However, for every $M \in (1,\infty)$ there exists a $\sigma$-finite setting $(X,\mu,\omega)$ and a function $f \in L^{p,\infty}_\mu(\ell^r_\omega)$ such that
	\begin{equation*}
	\norm{f}_{L^{p,\infty}_{\mu}(\ell^r_\omega)} \geq M K''.
	\end{equation*}
	In particular, let
	\begin{align*}
	X & = \{ j \in \N \colon 1 \leq j \leq m \}, \\
	\mu(A) & = 1, \qquad \qquad && \text{for every subset $A \subseteq X$, $A \neq \varnothing$,} \\
	\omega(A) & = \abs{A}, \qquad \qquad && \text{for every subset $A \subseteq X$,} \\
	f & = 1_{X}.
	\end{align*}
	Then, we have
	\begin{equation*}
	\norm{f}_{L^{p,\infty}_{\mu}(\ell^r_\omega)} = m^{\frac{1}{r}}, \qquad \qquad K'' \leq \sup \Big\{ \mu(\{ x_A \})^{\frac{1}{p}} \colon A \subseteq X, A \neq \varnothing \Big\}= 1, 
	\end{equation*}	
	where for every subset $A \subseteq X$, $A \neq \varnothing$ we choose $x_A \in A$ arbitrarily. Taking $m \in \N$ big enough we obtain the desired inequality. 
\end{remark}

\section{Function spaces on the collection of Heisenberg tiles} \label{sec:corollary}

\subsection*{$\sigma$-finite setting on the collection of Heisenberg tiles and lacunary sizes}

For all $m,l \in \Z$ we define the \emph{dyadic interval} $D(m,l)$ in $\R$ by
\begin{equation*}
D(m,l) \coloneq ( 2^l m, 2^l (m+1) ],
\end{equation*}
and the collection $\mathcal{D}$ of dyadic intervals in $\R$ by
\begin{equation*}
\mathcal{D} \coloneq \Big\{ D(m,l)  \colon m , l \in \Z \Big\}.
\end{equation*}
Moreover, for all $m,n,l \in \Z$ we define the \emph{dyadic Heisenberg tile} $H(m,n,l)$ in $\R^2$ by
\begin{equation*}
H(m,n,l) \coloneq D( m,l) \times D(n,-l),
\end{equation*}
and the collection $\mathcal{H}$ of dyadic Heisenberg tiles in $\R^2$ by
\begin{equation*}
\mathcal{H} \coloneq \Big\{ H(m,n,l)  \colon m ,n,l \in \Z \Big\}.
\end{equation*}

Then, for all $M,L \in \Z$ we define the \emph{stripe} $E(M,L) \subseteq \mathcal{H}$ by
\begin{equation*}
E(M,L) = E(D(M,L)) \coloneq \Big\{ H(m,n,l) \in \mathcal{H} \colon D(M,L) \subseteq D(m,l) \Big\}.
\end{equation*}
Moreover, for all $M,N,L \in \Z$ and for every $\kappa \in \Z$, $\kappa \geq 0$ we define the \emph{tree of $2^\kappa$-tiles} $T_{\kappa}(M,N,L) \subseteq \mathcal{H}$ by
\begin{equation*}
T_{\kappa}(M,N,L) \coloneq \Big\{ H(m,n,l) \in \mathcal{H} \colon D(m,l) \subseteq D(M,L), D(N,-L) \subseteq D(\lfloor n/{2^\kappa} \rfloor ,\kappa-l) \Big\},
\end{equation*}
where for every $x \in \R$ we define $\lfloor x \rfloor \in \Z$ to be the biggest integer number smaller than or equal to $x$.

Next, for every $\kappa \in \Z$, $\kappa \geq 0$ we define $(X_{\kappa},\mu_{\kappa},\omega_{\kappa})$ to be the $\sigma$-finite setting on the collection $\mathcal{H}$ of dyadic Heisenberg tiles as follows. Let
\begin{align*}
X_{\kappa} & = \mathcal{H}, \\
\mathcal{T}_{\kappa} & = \Big\{ T_{\kappa}(m,n,l) \colon m , n , l \in \Z \Big\}, \\
\sigma_{\kappa} (T_{\kappa}(m,n,l)) & = 2^{l} , && \text{for all $m , n , l \in \Z$,} \\
\omega_{\kappa} (H(m,n,l)) & = 2^{l} , && \text{for all $m , n , l \in \Z$,}
\end{align*}
and let $\mu_{\kappa}$ be the outer measures on $X_{\kappa}$ generated via minimal coverings by $\sigma_{\kappa}$ as follows. For every subset $A \subseteq X_{\kappa}$ we define $\mu_{\kappa}(A) \in [0,\infty]$ by
\begin{equation*}
\mu_{\kappa}(A) \coloneq \inf \Big\{ \sum_{T_{\kappa} \in \mathcal{A}} \sigma_{\kappa}(T_{\kappa}) \colon \mathcal{A} \subseteq \mathcal{T}_{\kappa}, A \subseteq \bigcup_{T_{\kappa} \in \mathcal{A}} T_{\kappa} \Big\},
\end{equation*}
where the infimum over an empty collection is understood to be $0$.

After that, for every $\kappa \in \Z$, $\kappa \geq 0$ we define a collection of tiles $T \subseteq \mathcal{H}$ a \emph{$\kappa$-lacunary tree} if it satisfies the following properties:
\begin{itemize}
\item There exists $M, N, L \in \Z$ such that
\begin{equation*}
T \subseteq T_\kappa(M,N,L).
\end{equation*}
\item For all $m,n,l,m',n',l' \in \Z$ such that 
\begin{equation*}
H(m,n,l), H(m',n',l') \in T, \qquad \qquad D(n,-l) \neq D(n',-l'),
\end{equation*}
we have
\begin{equation*}
D(n,-l) \cap D(n',-l') = \varnothing.
\end{equation*}
\end{itemize}
Moreover, we define $\mathcal{T}_{\kappa,\lac}$ to be the collection of all $\kappa$-lacunary trees. 

Finally, we define the \emph{lacunary size} $(S_{\kappa,\lac}, \mathcal{T}_{\kappa,\lac})$ as follows. For every $\kappa$-lacunary tree $T_{\kappa,\lac} \in \mathcal{T}_{\kappa,\lac}$ we define
\begin{equation*}
S_{\kappa,\lac} (f)(T_{\kappa,\lac}) \coloneq \mu_\kappa(T_{\kappa,\lac})^{-1} \norm{f 1_{T_{\kappa,\lac}}}_{L^1(X_\kappa,\omega_\kappa)} , 
\end{equation*}
and the outer Lorentz $L^{p,q}_{\mu_\kappa}(S_{\kappa,\lac})$ spaces as in the Introduction. In particular, for every $r \in (0,\infty]$ we denote by $S^r_{\kappa}$ the size $(S_{\kappa,\lac})_{r}$ defined in \eqref{eq:size_r} and in \eqref{eq:size_infty}.

It is easy to observe that the size $(S_{\kappa,\lac}, \mathcal{T}_{\kappa,\lac})$ satisfies the inequality in \eqref{eq:K_suport} with constant $K = 1$ and the inequality in \eqref{eq:new}.

\subsection*{The spaces $X^{p,q}_a(J,\kappa,\mathsf{size}_{2,\star})$ and equivalence with the outer $L^{p,\infty}_{\mu_{\kappa}}(S^2_\kappa)$ spaces}

The spaces $X^{p,q}_a(J,\kappa,\mathsf{size}_{2,\star})$ appearing in Section~3.8 in the article of Di Plinio and Fragkos \cite{2022arXiv220408051D} are associated with a triple of exponents $p,q,a \in [1,\infty]$, $p \geq a$, a dyadic interval $J \subseteq \R$, and a non-negative integer $\kappa \in \Z$, $\kappa \geq 0$. The size $( \mathsf{size}_{2,\star}, \mathcal{T}_{J,\kappa})$ on the collection of functions with support in the stripe $E(J)$ is defined as follows. Let $\mathcal{T}_{J,\kappa} \subseteq \mathcal{T}_{\kappa}$ be the collection of trees of $2^{\kappa}$-tiles contained in the stripe $E(J)$. For every tree $T_{\kappa} \in \mathcal{T}_{J,\kappa}$ and every function $F$ supported in the stripe $E(J)$ we define 
\begin{equation} \label{eq:cor_1}
\mathsf{size}_{2,\star} (F) (T_{\kappa}) \coloneq \norm{F 1_{T_{\kappa}}}_{L^\infty_{\mu_{\kappa}}(S^2_\kappa)}.
\end{equation}
The space $X^{p,q}_a(J,\kappa,\mathsf{size}_{2,\star})$ is defined by the following quasi-norm on functions $F$ supported in the stripe $E(J)$
\begin{equation} \label{eq:cor_2}
\norm{F}_{X^{p,q}_a(J,\kappa,\mathsf{size}_{2,\star})} \coloneq \sup \Big\{ \mu_{\kappa}(A)^{ \frac{1}{p} - \frac{1}{a}} \norm{F 1_A}_{L^{a,q}_{\mu_{\kappa}}(\mathsf{size}_{2,\star})} \colon A \in \Sigma_{\mu_{\kappa}} (E(J)) \Big\}.
\end{equation}

We are ready to state the equivalence between the $X^{p,q}_a(J,\kappa,\mathsf{size}_{2,\star})$ spaces and the outer $L^{p,\infty}_{\mu_{\kappa}}(S^2_\kappa)$ ones. 

\begin{corollary} \label{thm:corollary}
	For all $a,p,q \in [1,\infty]$, $p > a$ there exists a constant $C = C(a,p,q)$ such that the following properties hold true.
	
	For every $\kappa \in \Z$, $\kappa \geq 0$ let $(X_\kappa,\mu_\kappa,\omega_\kappa)$ be the setting and $\mathcal{T}_\kappa$ be the collection of generators described above. For every dyadic interval $J \in \mathcal{D}$, and every function $f$ on $X_\kappa$ we have
	\begin{equation*}
	\norm{f 1_{E(J)}}_{L^{p,\infty}_{\mu_\kappa}(S^2_\kappa)} = \alpha \sim_C \beta = \norm{f}_{X^{p,q}_a(J,\kappa,\mathsf{size}_{2,\star})} .
	\end{equation*}
	In particular, we have
	\begin{equation*}
	\alpha \leq \sqrt{\frac{p-a}{a}} \Big( \frac{p}{p-a} \Big)^{\frac{p}{2a}} \beta , \qquad \qquad \beta \leq 2^{\frac{1}{a}} \Big( \frac{p}{q (p-a)} \Big)^{\frac{1}{q}} \alpha.
	\end{equation*}
\end{corollary}

In fact, we can define analogous quasi-norms and spaces and prove an analogous statement in the case of the settings on the upper half $3$-space $\R^2 \times (0,\infty)$ described in the articles of Do and Thiele \cite{MR3312633} and of the author \cite{2023Fra}. The settings on $\mathcal{H}$ described above are discrete models of those on the upper half $3$-space.

\begin{proof}
The desired equivalence follows from Theorem~\ref{thm:intro_second}, Remark~\ref{rmk:suff} the definitions in \eqref{eq:cor_1} and in \eqref{eq:cor_2}, and the following auxiliary observation. For every dyadic interval $J \in \mathcal{D}$ and every function $F$ supported in $E(J)$ we have
\begin{equation*}
\norm{F}_{L^{\infty}_{\mu_\kappa}(S^2_\kappa)} = \norm{F}_{L^{\infty}_{\mu_\kappa}(\mathsf{size}_{2,\star})},
\end{equation*}
hence for all $p,q \in (0,\infty]$ and every function $F$ supported in $E(J)$ we have
\begin{equation*}
\norm{F}_{L^{p,q}_{\mu_\kappa}(S^2_\kappa)} = \norm{F}_{L^{p,q}_{\mu_\kappa}(\mathsf{size}_{2,\star})}.
\end{equation*}
\end{proof}

\appendix

\section{A proof of the inequality in \eqref{eq:Thm3.9}} \label{sec:appendix}

We prove the inequality in \eqref{eq:Thm3.9} originally stated in Proposition~3.9 in \cite{2022arXiv220408051D} applying standard interpolation results in the theory of outer $L^p$ spaces. 

\begin{proposition} [Proposition~3.9 in \cite{2022arXiv220408051D}] \label{thm:prop}
	For every $m \in \N$, $m \geq 2$ there exists a constant $C = C(m)$ such that the following property holds true. 
	
	Let $J \in \mathcal{D}$ be a dyadic interval in $\R$. Let $ a, p_1 \in (1,\infty)$, $ a \leq p_1$ and $p_2, \dots, p_m \in [1,\infty)$ satisfy
	\begin{equation*}
	\varepsilon \coloneq \Big( \sum_{j=1}^m \frac{1}{p_j} \Big) - 1 > 0.
	\end{equation*}
	Let the sizes $( \mathsf{size}_{2,\star}, \mathcal{T}_\kappa)$, $(s_2, \mathcal{T}_\kappa), \dots, (s_m, \mathcal{T}_\kappa)$ satisfy for every tree $T_{\kappa} \in \mathcal{T}_{\kappa}$ and for all functions $F_1, \dots, F_m$ supported in $E(J)$
	\begin{equation*}
	\ell^1_{\omega_\kappa} \Big( \prod_{j=1}^m F_j \Big) (T_\kappa) \leq \mathsf{size}_{2,\star} (F_1) (T_\kappa) \prod_{j=2}^m s_j (F_j) (T_\kappa).
	\end{equation*}	
	Then for all functions $F_1, \dots, F_m$ supported in $E(J)$ we have
	\begin{align*}
	\norm[\Big]{\prod_{j=1}^m F_j}_{L^1(E(J),\omega_\kappa)} \leq \frac{C a}{\varepsilon (a-1)} & \norm{F_1}_{X^{p_1,\infty}_a(J,\kappa,\mathsf{size}_{2,\star})} \times \\
	& \times \prod_{j=2}^m \max \{ \norm{F}_{L^{p_j,\infty}_{\mu_\kappa}(s_j)}, \norm{F}_{L^{\infty}_{\mu_\kappa}(s_j)} \}.
	\end{align*}	
\end{proposition}

We point out that in the article of Di Plinio and Fragkos \cite{2022arXiv220408051D} the outer measure $\mu_{\kappa}$ is renormalized with the length of the dyadic interval $J$, so a normalizing factor $\abs{J}^{-1}$ appears on the left hand side of the chain of inequalities in the statement of Proposition~3.9 in \cite{2022arXiv220408051D}.

To improve the readability of the proof, we omit explicit mention of $\kappa$ in the notation of $\mu_\kappa$ and $\omega_\kappa$.

\begin{proof}
First, a standard inequality between classical and outer $L^p$ quasi-norms (Proposition~3.6 in \cite{MR3312633}) yields
\begin{equation} \label{eq:appA_1}
\norm[\Big]{\prod_{j=1}^m F_j }_{L^1(E(J),\omega)} \leq \norm[\Big]{\prod_{j=1}^m F_j }_{L^{1,1}_\mu(\ell^1_\omega)}.
\end{equation}

Next, we recall outer H\"{o}lder's inequality for outer Lorentz $L^{p,q}$ spaces (Proposition~3.5 in \cite{2022arXiv220408051D}).
\begin{theorem} \label{thm:Holder_Lorentz}
	Let $(X,\mu,\omega)$ be a $\sigma$-finite setting. Let $s, s_1, \dots, s_m$ be sizes on $(X,\mu,\omega)$ such that for all measurable functions $F_1, \dots, F_m$ on $X$ and every measurable subset $A \in \Sigma$ we have
	\begin{equation} \label{eq:interp_condition}
	s \Big( \prod_{i=1}^m F_i \Big) (A) \leq \prod_{i=1}^m s_i(F_i) (A).
	\end{equation}
	Then for all $p, p_1, \dots, p_m \in (0,\infty]$ and all $q, q_1, \dots, q_m \in (0,\infty]$ such that
	\begin{equation*}
	\sum_{i=1}^m \frac{1}{p_i} = \frac{1}{p}, \qquad \qquad \sum_{i=1}^m \frac{1}{q_i} = \frac{1}{q},
	\end{equation*}
	we have
	\begin{equation*}
	\norm[\Big]{\prod_{i=1}^m F_i }_{L^{p,q}_\mu(s)} \leq m^{\frac{1}{p}} \norm{F_i }_{L^{p_i,q_i}_\mu(s_i)}.
	\end{equation*}
\end{theorem}
We apply the previous theorem with the sizes $s$ and $s_1$ defined by
\begin{equation*}
s = \ell^1_\omega, \qquad s_1 = S^2_{\kappa},
\end{equation*}
and the sizes $s_2, \dots, s_m$ satisfying the condition associated with the inequality in \eqref{eq:interp_condition}. Moreover, starting with the exponents $a, p_1, p_2, \dots, p_m$ satisfying the conditions in the statement of Proposition~\ref{thm:prop} we define the exponents $r_2, \dots, r_m \in (1,\infty)$ and $q_2, \dots, q_m \in (1,\infty) $ by
\begin{equation*}
P = \sum_{j=2}^m \frac{1}{p_j} = \frac{1}{p_1'} + \varepsilon, \qquad \qquad q_j = p_j P, \qquad \qquad r_j = a' q_j.
\end{equation*}
Therefore we obtain the inequality
\begin{equation} \label{eq:appA_2}
\norm[\Big]{\prod_{j=1}^m F_j }_{L^{1,1}_\mu(\ell^1_\omega)} \leq m \norm{F_1}_{L^{a,\infty}_\mu(s_1)} \prod_{j=2}^m \norm{F_j}_{L^{r_j,q_j}_\mu(s_j)}, 
\end{equation}

In particular, we observe that 
\begin{equation*}
P a' \geq \Big( \frac{1}{p_1'} + \varepsilon \Big) p_1' > 1,
\end{equation*}
hence for every $j \in \{2, \dots, m\}$ we have
\begin{equation*}
r_j > p_j,
\end{equation*}
hence the logarithmic convexity of the outer Lorentz $L^{p,q}$ spaces (Proposition~3.3 in \cite{MR3312633}) yields the inequality
\begin{equation} \label{eq:appA_3}
\norm{F_j}_{L^{r_j,q_j}_\mu(s_j)} \leq \Big( \frac{P}{\varepsilon} \Big)^{\frac{1}{q_j}} \norm{F_j}^{\frac{r_j - p_j}{r_j}}_{L^{\infty}_\mu(s_j)} \norm{F_j}^{\frac{p_j}{r_j}}_{L^{p_j,\infty}_\mu(s_j)}.
\end{equation}

Finally, the equivalence between the $X^{p,\infty}_a(J,\kappa,\mathsf{size}_{2,\star})$ spaces and the outer $L^{p,\infty}_\mu(S^2_\kappa)$ ones (Corollary~\ref{thm:corollary}) yields the inequality
\begin{equation} \label{eq:appA_4}
\norm{F}_{L^{p,\infty}_\mu(S^2_\kappa)} \leq \sqrt{\frac{p-a}{a}} \Big( \frac{p}{p-a} \Big)^{\frac{p}{2a}} \norm{F}_{X^{p,\infty}_a(J,\kappa,\mathsf{size}_{2,\star})} .
\end{equation}

The inequalities in \eqref{eq:appA_1} and in \eqref{eq:appA_2}--\eqref{eq:appA_4} yield the result stated in Proposition~3.9 in \cite{2022arXiv220408051D}. 
\end{proof}

\section{Proofs of Lemma~\ref{thm:lemma_2} and Lemma~\ref{thm:lemma_1}} \label{sec:appendix_2}

\begin{proof} [Proof of Lemma~\ref{thm:lemma_2}]
	Let $\varepsilon > 0$. By the existence of a measurable subset satisfying the inequality in \eqref{eq:condition_3}, for every $\lambda \in (0,\rho)$ we have
	\begin{equation*}
	\mu(\ell^r_\omega(f) > \lambda) > 0.
	\end{equation*}
	For every $\lambda \in (0,\rho)$ we choose a measurable subset $\Omega_\lambda \in \Sigma$ such that
	\begin{equation*}
	\norm{f 1_{\Omega_\lambda^c} }_{L^\infty_\mu(\ell^r_\omega)} \leq \lambda, \qquad \qquad \mu(\Omega_\lambda) \leq (1+\varepsilon) \mu(\ell^r_\omega(f) > \lambda).
	\end{equation*}
	We have
	\begin{align*}
	\norm{f}^r_{L^\infty_\mu(\ell^r_\omega)} \mu(\Omega_\lambda) & \geq \norm{f 1_{\Omega_\lambda} }^r_{L^r(X,\omega)} \\
	& \geq \norm{f 1_{A} }^r_{L^r(X,\omega)} - \norm{f 1_{A} 1_{\Omega_\lambda^c} }^r_{L^r(X,\omega)} \\
	& \geq \mu(A) ( \rho^r - \lambda^r ).
	\end{align*}
	Taking $\varepsilon > 0$ arbitrarily small we obtain the desired inequality.
\end{proof}

\begin{proof}[Proof of Lemma~\ref{thm:lemma_1}]
	We argue by contradiction and we assume that there exists no measurable subset $B \in \Sigma$ such that
	\begin{equation*}
	\ell^r_\omega(f) (B) \geq \lambda,
	\end{equation*}
	hence
	\begin{equation*}
	\lambda \geq \norm{f}_{L^\infty_\mu(\ell^r_\omega)},
	\end{equation*}
	This yields a contradiction with the inequality in \eqref{eq:condition}.
	
	Therefore, there exists a measurable subset $B \in \Sigma$ satisfying the first inequality in \eqref{eq:condition_2}. We define the collection $\{ B_n \colon n \in \N_\lambda \} \subseteq \Sigma$ of measurable subsets by a forward recursion on $n \in \N_\lambda$, where $\N_\lambda$ is a finite non-empty initial string of $\N$. We define the subset $A_0 = \varnothing$ and for all $n \in \N_\lambda$ we define the measurable subset $A_n \in \Sigma$ by
	\begin{equation*}
	A_n = \bigcup_{m \leq n} B_m,
	\end{equation*}
	and the collection $\mathcal{B}_n \subseteq \Sigma$ of measurable subsets of $X$ by
	\begin{equation*}
	\mathcal{B}_n = \Big\{ B \in \Sigma \colon \ell^r_\omega(f 1_{A^c_{n-1}})(B) \geq \lambda \Big\}.
	\end{equation*}
	In particular, we have $\mathcal{B}_1 \neq \varnothing$. 
	
	If $\mathcal{B}_n$ is empty, we define $\N_\lambda \subseteq \N$ by
	\begin{equation*}
	\N_\lambda = \{ 1, \dots, n-1 \}.
	\end{equation*}
	If $\mathcal{B}_n$ is not empty, by Lemma~\ref{thm:lemma_2} for every measurable subset $B \in \mathcal{B}_n$ we have
	\begin{equation} \label{eq:app_4}
	\mu(B) + \sum_{m \in \N, m < n} \mu(B_m) \leq C \mu(\ell^r_\omega(f) \geq \lambda/2) \leq C 2^{p} \lambda^{-p} \norm{f}^p_{L^{p,\infty}_\mu(\ell^r_\omega)} < \infty,
	\end{equation}
	hence there exists $j = j(n) \in \Z$ such that
	\begin{equation*}
	\sup \Big\{ \mu(B) \colon B \in \mathcal{B}_n \Big\} \in (2^j, 2^{j+1}].
	\end{equation*}
	We choose $B_n \in \mathcal{B}_n$ such that
	\begin{equation*}
	\mu(B_n) \in (2^j,2^{j+1}].
	\end{equation*}
	By the inequality in \eqref{eq:app_4} there exists $N = N(r,\lambda,\norm{f}_{L^{\infty}_\mu(\ell^r_\omega)}) \in \N$ such that
	\begin{equation*}
	j(n+N) \leq j(n) - 1,
	\end{equation*}
	We argue by contradiction and we assume there exists a collection $\{ B_{n + i} \colon i \in \{0, \dots, N \} \} \subseteq \Sigma$ of pairwise disjoint measurable subsets such that for every $i \in \{ 0, 1, \dots, N \}$ we have
	\begin{equation*}
	\ell^r_\omega(f)(B_{n+i}) \geq \lambda, \qquad \qquad \mu(B_{n+i}) \in (2^{j(n)},2^{j(n)+1}].
	\end{equation*}
	Then we have
	\begin{equation*}
	\norm[\Big]{ f 1_{\bigcup_{i = 0}^N B_{n+i} } }^r_{L^r(X,\omega)} \geq \lambda^r \sum_{i = 0}^N \mu(B_{n+i}) \geq \lambda^r (N+1) 2^{j(n)},
	\end{equation*}
	and by the definition of $j(n)$ we have
	\begin{equation*}
	\mu \Big( \bigcup_{i = 0}^N B_{n+i} \Big) \leq 2^{j(n)+1},
	\end{equation*}
	hence
	\begin{equation*}
	\ell^r_\omega( f ) \Big( \bigcup_{i = 0}^N B_{n+i} \Big)  \geq \Big( \frac{N+1}{2} \Big)^{\frac{1}{r}} \lambda > \norm{f}_{L^\infty_\mu(\ell^r_\omega)},
	\end{equation*}
	yielding a contradiction for $N$ too big.
	
	Therefore the sequence $\{ j(n) \colon n \in \N_\lambda \}$ is non-increasing and we have that either $\N_\lambda$ is finite or it is infinite and
	\begin{equation*}
	\inf \Big\{ j(n) \colon n \in \N_\lambda \Big\} = -\infty.
	\end{equation*}
	Defining the measurable subset $B \in \Sigma$ by
	\begin{equation*}
	B = \bigcup_{n \in \N_\lambda} B_n,
	\end{equation*}
	we obtain the desired inequality.
\end{proof}

\bibliographystyle{amsplain}
\bibliography{mybibliography}

\end{document}